\documentclass[11pt]{amsart}%
\usepackage{hyperref}
\usepackage{amsmath}%
\usepackage{amsaddr}
\usepackage{amsfonts}%
\usepackage{amssymb}%
\usepackage{amsthm}
\usepackage{mathrsfs}
\usepackage{comment}
\usepackage{todonotes}
\usepackage{bbm}
\usepackage{amscd}
\usepackage{youngtab}
\usepackage{tikz}
\usepackage{bm}
\usepackage{enumitem}
\usetikzlibrary{arrows}
\usetikzlibrary{matrix}

\newcommand{\ph}{\varphi}

\newcommand{\Q}{\mathbb{Q}}

\newcommand{\Z}{\mathbb{Z}}

\renewcommand{\phi}{\varphi}

\newcommand{\N}{\mathbb{N}}

\newcommand{\sur}{\twoheadrightarrow}

\DeclareMathOperator{\im}{im}

\renewcommand{\hat}{\widehat}

\newtheorem{theorem}{Theorem}[section]
\newtheorem{def-prop}[theorem]{Definition-Proposition}
\newtheorem{prop}[theorem]{Proposition}
\newtheorem{conj}[theorem]{Conjecture}
\newtheorem{lemma}[theorem]{Lemma}

\theoremstyle{definition}
\newtheorem{ex}[theorem]{Example}

\theoremstyle{remark}
\newtheorem{remark}{Remark}

\begin{document}

\title{Path counting and rank gaps in differential posets}
\author{Christian Gaetz}
\address{Department of Mathematics, Massachusetts Institute of Technology, Cambridge, MA.}
\email{\href{mailto:gaetz@mit.edu}{gaetz@mit.edu}} 
\thanks{C.G. was supported by the National Science Foundation Graduate Research Fellowship under Grant No. 1122374.}

\author{Praveen Venkataramana}
\address{Division of Physics, Mathematics, and Astronomy, California Institute of Technology, Pasadena, CA.}
\email{\href{mailto:pvenkata@caltech.edu}{pvenkata@caltech.edu}}
\date{\today}

\begin{abstract}
We study the gaps $\Delta p_n$ between consecutive rank sizes in $r$-differential posets by introducing a projection operator whose matrix entries can be expressed in terms of the number of certain paths in the Hasse diagram.  We strengthen Miller's result that $\Delta p_n \geq 1$, which resolved a longstanding conjecture of Stanley, by showing that $\Delta p_n \geq 2r$.  We also obtain stronger bounds in the case that the poset has many substructures called \textit{threads}. 
\end{abstract}


\maketitle

\section{Introduction} \label{sec:intro}

Differential posets are a class of partially ordered sets, originally defined by Stanley \cite{Stanley1988}, which generalize many of the enumerative and combinatorial properties of Young's lattice $Y$, the lattice of integer partitions.  We refer the reader to \cite{Stanley2012} for definitions and basic facts about posets.

If $P$ is a graded poset and $S \subseteq P$, let $\Q S$ denote the $\Q$-vector space with basis $S$.  Let $P_n$ denote the $n$-th rank of $P$ and define the \textit{up and down operators} $U_n: \Q P_n \to \Q P_{n+1}$ and $D_n: \Q P_n \to \Q P_{n-1}$ by
\begin{align*}
U_n x &= \sum_{x \lessdot y} y \\
D_n x &= \sum_{z \lessdot x} z
\end{align*}
where $x \in P_n$ and $\lessdot$ denotes the covering relation in $P$.  We often omit the subscripts and write $U,D$ when no confusion can result.  For $r \in \Z_{>0}$ a locally-finite $\N$-graded poset $P$ with $\hat{0}$ is an \textit{r-differential poset} if, for all $n \geq 0$ we have
\begin{align}
D_{n+1}U_n - U_{n-1}D_n = r \cdot I 
\label{eq:DU-UD}
\end{align}
as linear operators.

\begin{ex} \label{ex:Y-and-Z}
Young's lattice $Y$ is a 1-differential poset (see Figure \ref{fig:Y}).  More generally the product poset $Y^r$ is an $r$-differential poset.  

Another family of examples are the Fibonacci lattices $Z(r)$ (see Figure \ref{fig:Z2}).  They are produced by the following iterative procedure:
\begin{enumerate}
\item Having defined a differential poset $P$ up to rank $n$, add an element $\hat{x}$ to $P_{n+1}$ for each element $x \in P_{n-1}$ so that $\hat{x}$ covers exactly those $y \in P_n$ which cover $x$.  
\item For each $y \in P_n$, add an additional $r$ singletons covering $y$.
\end{enumerate}
The $r$-differential poset $Z(r)$ is obtained by iterating this procedure beginning with a single point \cite{Stanley1988}.
\end{ex}

\begin{center}
\begin{figure} 
\begin{tikzpicture} [every node/.style={draw,shape=circle,fill=black,scale=0.4}, scale=1.4]
\node(a0) at (0,0) {};
\node[fill=white](b0) at (-0.7,1) {};
\node[fill=white](b1) at (0.7,1) {};
\node(c0) at (0,2) {};
\node[fill=white](c1) at (-1.4,2) {};
\node[fill=white](c2) at (-.7,2) {};
\node[fill=white](c3) at (.7,2) {};
\node[fill=white](c4) at (1.4,2) {};
\node(d0) at (-0.7,3) {};
\node(d1) at (0.7,3) {};

\node[fill=white](d2) at (-1.4,3) {};
\node[fill=white](d3) at (-2.1,3) {};
\node[fill=white](d4) at (-1,3) {};
\node[fill=white](d5) at (-.4,3) {};
\node(d6) at (-.2,3) {};
\node(d7) at (.2,3) {};
\node[fill=white](d8) at (1,3) {};
\node[fill=white](d9) at (.4,3) {};
\node[fill=white](d10) at (1.4,3) {};
\node[fill=white](d11) at (2.1,3) {};

\draw (a0)--(b0);
\draw (a0)--(b1);
\draw (b0)--(c0);
\draw (b1)--(c0);
\draw (b0)--(c1);
\draw (b0)--(c2);
\draw (b1)--(c3);
\draw (b1)--(c4);
\draw (c0)--(d0);
\draw (c1)--(d0);
\draw (c2)--(d0);
\draw (c0)--(d1);
\draw (c3)--(d1);
\draw (c4)--(d1);

\draw (c1)--(d2);
\draw (c1)--(d3);

\draw (c2)--(d4);
\draw (c2)--(d5);

\draw (c3)--(d8);
\draw (c3)--(d9);

\draw (c4)--(d10);
\draw (c4)--(d11);

\draw (c0)--(d6);
\draw (c0)--(d7);
\end{tikzpicture}
\caption{The Hasse diagram of the Fibonacci lattice $Z(2)$, a 2-differential poset, up to rank 3.  The thread elements (see Section \ref{sec:threads}) are unfilled.}
\label{fig:Z2}
\end{figure}
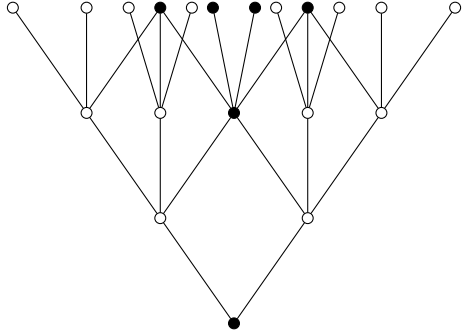
\end{center}

\begin{center}
\begin{figure} 
\begin{tikzpicture} [every node/.style={draw,shape=circle,fill=black,scale=0.4}, scale=1.4]
\node(zero) at (0,-1) {};
\node[fill=white](a0) at (0,0) {};
\node[fill=white](b0) at (-0.7,1) {};
\node[fill=white](b1) at (0.7,1) {};
\node(c0) at (0,2) {};
\node[fill=white](c1) at (-1.4,2) {};
\node[fill=white](c4) at (1.4,2) {};
\node(d0) at (-0.7,3) {};
\node(d1) at (0.7,3) {};

\node(d2) at (-.7,3) {};
\node[fill=white](d3) at (-2.1,3) {};
\node(d4) at (0,3) {};
\node(d10) at (.7,3) {};
\node[fill=white](d11) at (2.1,3) {};

\node[fill=white](e0) at (-2.8,4) {};
\node(e1) at (-1.4,4) {};
\node(e2) at (-.7,4) {};
\node(e3) at (0,4) {};
\node(e4) at (.7,4) {};
\node(e5) at (1.4,4) {};
\node[fill=white](e6) at (2.8,4) {};

\draw (d0)--(e2);
\draw (d0)--(e3);
\draw (d0)--(e1);
\draw (d4)--(e2);
\draw (d4)--(e4);
\draw (d1)--(e3);
\draw (d1)--(e4);
\draw (d1)--(e5);
\draw (d3)--(e0);
\draw (d3)--(e1);
\draw (d11)--(e5);
\draw (d11)--(e6);

\draw (zero)--(a0);
\draw (a0)--(b0);
\draw (a0)--(b1);
\draw (b0)--(c0);
\draw (b1)--(c0);
\draw (b0)--(c1);
\draw (b1)--(c4);
\draw (c0)--(d0);
\draw (c1)--(d0);
\draw (c0)--(d1);
\draw (c0)--(d4);
\draw (c4)--(d1);
\draw (c1)--(d2);
\draw (c1)--(d3);
\draw (c4)--(d10);
\draw (c4)--(d11);
\end{tikzpicture}
\caption{The Hasse diagram of Young's lattice $Y$, a 1-differential poset, up to rank 5.  The thread elements (see Section \ref{sec:threads}) are unfilled.}
\label{fig:Y}
\end{figure}
\end{center}

One of the most basic properties of differential posets that one might hope to study are the possible \textit{rank functions} $p_n=|P_n|$, not least because the rank functions for the examples $Y$ and $Z(1)$ from Example \ref{ex:Y-and-Z} are the integer partition numbers and Fibonacci numbers respectively, two sequences of integers which have received an immense amount of study.  The first general result obtained about $p_n$ was the following:  

\begin{prop}[Stanley \cite{Stanley1988}] \label{prop:weak-growth}
For any differential poset $P$ the rank sizes $p_0,p_1,p_2,...$ weakly increase.
\end{prop}

Two main lines of research about the rank function have developed, both motivated by the intuition that $Y^r$ should be the ``smallest" $r$-differential poset, and $Z(r)$ the ``largest".

The first concerns the asymptotic behavior of $p_n$.  Stanley and Zanello \cite{Stanley2012a} have shown that in an $r$-differential poset,
\[
p_n \gg n^a \cdot e^{\sqrt{2rn}}
\]
a bound which is fairly close to the well-known asymptotics for $p_r(n)=|(Y^r)_n|$:
\[
p_r(n) \gg c \cdot n^b \cdot e^{\pi \sqrt{2rn/3}}
\]
where $a,b,c$ are constants.  On the other hand, Byrnes (\cite{Byrnes2012}, Theorem 1.2) has shown that for any $r$-differential poset
\begin{equation*} 
p_n \leq |Z(r)_n|.  
\end{equation*}
From the recursive definition of $Z(r)$ in Example \ref{ex:Y-and-Z}
it is immediate that $|Z(r)|_n=r \cdot |Z(r)|_{n-1}+|Z(r)|_{n-2}$.  Standard techniques allow us to solve this recurrence and observe that 
\begin{align} 
|Z(r)|_n &= \frac{r+\sqrt{r^2+4}}{2\sqrt{r^2+4}}\cdot \ph_r^n+\frac{\sqrt{r^2+4}-r}{2\sqrt{r^2+4}}\cdot \psi_r^n \nonumber \\
&\leq 2 \cdot \ph_r^n \label{eq:pn-upper-bound}
\end{align}
where $\ph_r=(r+\sqrt{r^2+4})/2$ and $\psi_r=(r-\sqrt{r^2+4})/2$.

The second, more difficult, line of research has sought explicit bounds on the \textit{rank gaps} $\Delta p_n = p_n-p_{n-1}$, which, by Proposition \ref{prop:weak-growth} are nonnegative. These rank gaps are the eigenvalue multiplicities of $DU_n$ (see \cite{Stanley1988}) and also appear in the study of the finer algebraic structure of differential posets (see \cite{Agarwal2017, Miller2009}).  It was conjectured in 1988 in Stanley's original paper \cite{Stanley1988} that, except when $r=n=1$, the rank gaps $\Delta p_n$ are strictly positive.  This conjecture was proven by Miller only as recently as 2013, who showed:

\begin{theorem}[Miller \cite{Miller2013}] \label{thm:strict-growth}
For any $r$-differential poset, $\Delta p_n \geq 1$, unless $r=n=1$.
\end{theorem}  

Despite the fact that Stanley's conjecture was open for so long, it seems that much stronger results may be true.  For example, Stanley and Zanello \cite{Stanley2012a} have asked whether \emph{all} iterated partial differences $\Delta^t p_n$ are eventually positive for large $n$, and whether this is already true for $n=2$ in the case $t=r$.  One implication of a positive answer to this question would be that
\begin{equation} 
\lim_{r \to \infty} \min_{\substack{P \text{ $r$-differential} \\ n\geq 2}} \Delta p_n = \infty \label{eq:r-limit}
\end{equation}
and that 
\begin{equation}
\lim_{n \to \infty} \Delta p_n = \infty \label{eq:n-limit}
\end{equation}
for any differential poset $P$.

Our main theorem is the first improvement on Miller's bound for general differential posets, and resolves (\ref{eq:r-limit}).  In addition, it seems plausible that the new methods we use to prove this result might lead to a resolution of (\ref{eq:n-limit}), see Remark \ref{rem:strategy}.

\begin{theorem} \label{thm:main-theorem}
For any $r$-differential poset, $\Delta p_n \geq 2r$ for $n \geq 4$ if $r=1$, $n \geq 3$ if $r=2$ and for $n \geq 2$ if $r>2$.
\end{theorem}

In Theorem \ref{thm:gaps-from-threads} we also prove a stronger result in the case where $P$ has many \textit{threads} (see Section \ref{sec:threads}).

\section{Projection matrices and path counting}
Let $\langle, \rangle$ denote the inner product on $\Q P$ which makes the elements of $P$ into an orthonormal basis.  It is easy to see that the operators $U,D$ are adjoint with respect to this inner product:
\[
\langle Ux,y \rangle = \langle x, Dy \rangle.
\]
Formulas for counting paths in the Hasse diagram of $P$ (viewed as a graph) can often be expressed in terms of the inner product.  For example $e(x):=\langle U^n \hat{0}, x \rangle$ for $x \in P_n$ is easily seen to be the number of paths in the Hasse diagram of $P$ from $\hat{0}$ to $x$ using only upward steps.  We will make use of the following enumerative identity:

\begin{prop}[Stanley \cite{Stanley1988}] \label{prop:sum-of-squares}
Let $P$ be an $r$-differential poset, then 
\[
\sum_{x \in P_n} e(x)^2 = r^n n!.
\]
\end{prop}

The following linear algebraic facts will also be used:

\begin{prop}[Stanley \cite{Stanley1988}]
\label{prop:U-injective}
Let $P$ be a differential poset and $n \geq 0$, then:
\begin{itemize}
\item[(a)] $U_n: \Q P_n \to \Q P_{n+1}$ is injective, and
\item[(b)] $DU_n$ is invertible.
\end{itemize}
\end{prop}

Our main tool will be to study the projection operator onto $\im(U)$, whose entries can be expressed in terms of path counting in the Hasse diagram of $P$.

\begin{prop} \label{prop:M-is-projection}
Let $M_n=U(DU)^{-1}D_n$, then:
\begin{itemize}
\item[(a)] $M$ is the orthogonal projection $\Q P_n \sur \im(U_{n-1})$.
\item[(b)] In the standard basis $P_n$ for $\Q P_n$, the entries for $M_n=(m_{xy})_{x,y \in P_n}$ are given by
\[
m_{xy}=\sum_{k=1}^n (-1)^{k-1} \frac{\langle D^k x, D^k y \rangle}{r^k k!}.
\]
Futhermore $\langle D^k x, D^k y \rangle$ counts the number of pairs of paths of length $k$ beginning at $x$ and $y$ and ending at a common element of $P_{n-k}$.
\end{itemize}
\end{prop}
\begin{proof} We have 
\[
M^2=U(DU)^{-1}DU(DU)^{-1}D=M,
\]
so $M$ is a projection.  Since $D$ and $(DU)^{-1}$ are surjective by Proposition \ref{prop:U-injective} ($D$ is the transpose of an injective map), $\im(M_n)=\im(U_{n-1})$, so $M$ is a projection onto $\im(U_{n-1})$.  Orthogonality follows easily from the adjointness of $U,D$, so (a) is proven.  For (b), we first show that 
\begin{align*}
(DU_{n-1})^{-1} = \sum_{k=0}^{n-1} (-1)^k \frac{U^kD^k}{r^{k+1}(k+1)!}.
\end{align*}
Indeed, let $R$ denote the right-hand side, then:
\begin{align*}
DU_{n-1} \cdot R &= \sum_{k=0}^{n-1} (-1)^k \frac{DUU^kD^k}{r^{k+1}(k+1)!} \\
&= \sum_{k=0}^{n-1} (-1)^k \frac{U^{k+1}D^{k+1}+r(k+1)U^kD^k}{r^{k+1}(k+1)!} \\
&= I + (-1)^{n-1} \frac{U^nD^n}{r^nn!} \\
&= I.
\end{align*}
where the second equality follows from applying (\ref{eq:DU-UD}), the third from collapsing the telescoping sum, and the last from observing that $D^n=0$ on $\Q P_{n-1}$.  Now, by definition we have 
\[
M_n = \sum_{k=0}^{n-1} (-1)^k \frac{U^{k+1}D^{k+1}}{r^{k+1}(k+1)!}.
\]
Reindexing and taking matrix entries by applying adjointness, we obtain the desired result.
\end{proof}

\begin{remark}
The formula for $(DU_{n+1})^{-1}$, which plays a key role in the proof of Proposition 2.3, was suggested to us in personal communication by Alexander Miller, who more generally gave a formula for $(UD_{n+1}+kI)^{-1}$ (our formula is the case where $k=r$).  The expression for $m_{xy}$ is also similar to an expression obtained by Miller in \cite{Miller2013} for some of the entries in $(DU+kI)^{-1}$.  Motivated by an earlier representation-theoretic proof of strict rank growth in the case of \textit{differential towers of groups} \cite{Miller2009} (see also \cite{Gaetz2018}), Miller studied the denominators in $(DU+kI)^{-1}$, rather than the eigenvalues of submatrices of $M$, as we do.
\end{remark}

We will always think of $M_n$ as a matrix in the standard basis for $\Q P_n$.  Given $S \subseteq P_n$, let $M_S$ denote the principal submatrix of $M_n$ with rows and columns indexed by the elements of $S$.

\begin{lemma} \label{lem:submatrix}
Let $S \subseteq P_n$ and suppose $M_S$ has no eigenvalue equal to 1, then $\Delta p_n \geq |S|$.
\end{lemma}
\begin{proof}
If $M_S$ has no eigenvalue equal to 1, then $\Q S$ must intersect $\im(U_{n-1})$ trivially.  Since $U_{n-1}$ is injective, this means that 
\[
\dim(\Q P_n) \geq \dim(\Q P_{n-1}) + \dim(\Q S).
\]
And so $\Delta p_n \geq |S|$.
\end{proof}

\section{Thread elements and rank gaps}
\label{sec:threads}

It seems very difficult to estimate the entries $m_{xy}$ for arbitrary $x,y \in P_n$ using the formula in Proposition \ref{prop:M-is-projection}, since the sum is alternating and individual terms can be much larger than the sum (it follows from the fact that $M^2=M$ that $|m_{xy}| \leq 1$, however individual terms in the sum can be much larger than 1).  We therefore define special elements for which $m_{xy}$ can be more easily analyzed.

An element $x \in P$ is called a \textit{singleton} if $x= \hat{0}$ or $x$ covers a unique element of $P$.  A \textit{thread element} is a singleton $x \in P$ which covers another singleton.  A key property of thread elements is the following:

\begin{prop}[Miller and Reiner \cite{Miller2009}]
\label{prop:threads-extend}
Let $t_0 \in P$ be a thread element.  Then there exist thread elements 
\[
t_0 \lessdot t_1 \lessdot t_2 \lessdot \cdots
\]
We call this infinite sequence a thread.
\end{prop}

Given a differential poset $P$, let $\tau_n$ denote the number of thread elements in $P_n$.  Proposition \ref{prop:threads-extend} implies that $\tau$ is a nondecreasing function of $n$.

\begin{ex}
If $P=Y^r$, then the thread elements in ranks $n \geq 2$ are the elements $(\emptyset,...,\emptyset, \lambda, \emptyset,...,\emptyset)$ where $\lambda=(n)$ or $(1^n)$.  Thus $\tau_n=2r$ for $n \geq 2$.

If $P=Z(r)$, we get $r$ thread elements in rank $n$ for each singleton in $P_{n-1}$, and there are $r$ of these singletons for each element of $P_{n-2}$.  Thus $\tau_n=r^2 p_{n-2}$; in particular, $\tau$ grows exponentially in $n$ by (\ref{eq:pn-upper-bound}).
\end{ex}

\begin{theorem} \label{thm:gaps-from-threads}
Let $P$ be an $r$-differential poset.  Then for $N \geq 4n$:
\[
\Delta p_N \geq \tau_n.
\]
\end{theorem}
\begin{proof}
Assume $n \geq 1$, since $\tau_0=0$.  Let $T \subset P_n$ be the set of thread elements in $P_n$, so $|T|=\tau_n$.  Suppose $N \geq 4n$ and for each element $t \in T$, extend $t$ to a thread via Proposition \ref{prop:threads-extend}, and let $\hat{t}$ be the element of this thread at rank $N$.  It is clear from the definition of threads that $\hat{s} \neq \hat{t}$ for $s \neq t$.  Let $\hat{T}=\{\hat{t}\:| \: t \in T\}$.  We now want to show that $M_{\hat{T}}$ has no eigenvalue equal to 1, and apply Lemma \ref{lem:submatrix}.

Let $s,t \in T$.  We first bound the tail of the sum formula for $m_{\hat{s} \hat{t}}$ given in Proposition \ref{prop:M-is-projection}.  We have
\begin{align}
\left| \sum_{k=N-n+1}^N (-1)^{k-1} \frac{\langle D^k \hat{s}, D^k \hat{t} \rangle}{r^k k!} \right| &= \left| \sum_{k=N-n+1}^N (-1)^{k-1} \frac{\langle D^{k-N+n} s, D^{k-N+n} t \rangle}{r^k k!} \right| \label{eq:using-thread} \\
& \leq \sum_{k=N-n+1}^N \frac{r^n n!}{r^k k!} \label{eq:using-e} \\
& \leq n \cdot \frac{r^n n!}{r^{3n} (3n)!}. \label{eq:a(r,n)}
\end{align}
The equality (\ref{eq:using-thread}) follows because $\hat{s}, \hat{t}$ are thread elements, so $D^{N-n} \hat{s}=s$ and $D^{N-n}\hat{t}=t$.  The inequality (\ref{eq:using-e}) uses Proposition \ref{prop:sum-of-squares}: the number of paths down from $s$ to $\hat{0}$ is at most $\sqrt{r^nn!}$, and thus the number of paths down from $s$ and ending in any particular rank is at most this number; as this also holds for $t$, the bound follows.  Letting $a(r,n)$ denote the expression on the right-hand side of (\ref{eq:a(r,n)}), all off-diagonal entries of $M_{\hat{T}}$ satisfy
\[
|m_{\hat{s} \hat{t}}| \leq a(r,n),
\]
since $\hat{s}$ and $\hat{t}$ lie on separate threads and have no common lower bounds above rank $n$.  The diagonal entries satisfy
\begin{align*}
m_{\hat{t} \hat{t}} &\leq \left( \frac{1}{r}-\frac{1}{2r^2}+ \cdots \pm \frac{1}{r^{N-n}(N-n)!} \right) + a(r,n) \\ 
&< \left(\frac{1}{r}-\frac{1}{2r^2}+\frac{1}{6r^3} \right) + a(r,n) \\
&\leq \frac{2}{3} + a(r,n).
\end{align*}
Let $\phi_r=\frac{1}{2}(r+\sqrt{r^2+4})$.  By Byrnes' result (\ref{eq:pn-upper-bound}), $p_n \leq 2 \phi_r^n$, and so $\hat{T}$ has at most this size.  By Gershgorin's Circle Theorem, in order to see that $M_{\hat{T}}$ has no eigenvalue equal to 1, it suffices to show that
\[
\frac{2}{3} + 2 \phi_r^n a(r,n) \leq 1.
\]
If $r>1$, then $r^{2n} > \phi_r^n$ and $2nn!/(3n)! \leq 1/3$.  If $r=1$, we have $\tau_1=1$, so Theorem \ref{thm:strict-growth} gives the result when $n=1$; for $n\geq 2$ it is easy to see that $2 \phi_1^n a(1,n)<1/3$.
\end{proof}

We now give the proof of Theorem \ref{thm:main-theorem}.

\begin{prop}
\label{prop:2r-threads}
Let $P$ be an $r$-differential poset.  Then at least two thread elements cover each $x \in P_1$.
\end{prop}
\begin{proof}
In any $r$-differential poset $P$, Equation (\ref{eq:DU-UD}) implies that there are $r$ elements in $P_1$, and all of these are singletons, so they must each be covered by $r+1$ elements.  Given $x \in P_1$, for each $y$ in $P_1$ different from $x$, there is at most one element covering both $x,y$, as can easily be seen from the defining relation (\ref{eq:DU-UD}).  Thus there are at most $r-1$ elements covering $x$ which are not singletons, and so at least two thread elements cover $x$.
\end{proof}

\begin{proof}[Proof of Theorem \ref{thm:main-theorem}]
This proof is similar to the proof of Theorem \ref{thm:gaps-from-threads}, except that we can be more precise, since the threads involved begin in ranks 0 or 1.  We prove the case $r \geq 3$ here, the other cases are nearly identical, except that the different values of $r$ in the denominators imply that larger ranks $n$ must be considered in order for the Gershgorin bound to hold.

Let $r \geq 3, n \geq 2$ and let $T=\{t_1,...,t_r,s_1,...,s_r\} \subset P_2$ be a set of $2r$ thread elements, with $t_i,s_i$ covering each element $x_i$ of $P_1$; such elements exist by Proposition \ref{prop:2r-threads}.  Extend each $t \in T$ to an infinite thread, and let $\hat{T}=\{\hat{t} \: | \: t \in T\}$ be the elements of rank $n$ in these threads.  Then $M_{\hat{T}}$ has diagonal elements:
\[
m_{\hat{t} \hat{t}} = \sum_{k=0}^{n-1} (-1)^k \frac{\langle D^{k+1}\hat{t},D^{k+1}\hat{t}\rangle}{r^{k+1}(k+1)!} = \frac{1}{r} - \frac{1}{2!r^2} + \cdots + (-1)^{n-1}\frac{1}{n!r^n}  < \frac{1}{r} \leq \frac{1}{3}
\]

for some $k$. For off-diagonal entries we have
\[
|m_{\hat{t_i}, \hat{s_i}}| = \left|\frac{1}{r^{n-1}(n-1)!} - \frac{1}{r^n n!}\right| \leq \frac{1}{r^{n-1}(n-1)!} \leq \frac{1}{3},
\]
and all other off-diagonal entries satisfy
\[
|m_{\hat{s} \hat{t}}|=\left| \frac{1}{r^n n!} \right| \leq \frac{1}{2r^2}.
\]
The sum of the absolute values of the matrix entries of $M_{\hat{T}}$ across a row is thus at most $1/3+1/3+(2r-2)/2r^2 < 1$.  Therefore by the Gershgorin Circle Theorem, $M_{\hat{T}}$ has no eigenvalue equal to 1.  Applying Lemma \ref{lem:submatrix} completes the proof.

\end{proof}

\begin{remark} \label{rem:strategy}
One can obtain stronger bounds in an ad-hoc manner using Lemma \ref{lem:submatrix}.  For example, it is possible to show that when $r=1$ we have $\Delta p_n \geq 3$ for $n \geq 6$ by letting $s_n, t_n$ be the elements of rank $n$ in the two threads which begin in rank 2, and letting $u_n$ be a certain sequence of elements which are ``close" to $s_n$ so that the entries of $M_{\{s_n,t_n,u_n\}}$ can be well approximated by exhaustively considering all possible paths.  What is needed to progress further on this problem is identify families $S_n \subset P_n$ whose size grows with $n$ for which the entries of $M_{S_n}$ can be estimated.
\end{remark}

\begin{conj}
In any differential poset $\lim_{n \to \infty} \Delta p_n = \infty$.
\end{conj}

\section*{Acknowledgements}
The authors wish to thank Fabrizio Zanello for helping to initiate this joint project and Richard Stanley for his helpful conversations.  We are also grateful to Patrick Byrnes for making his computer code available and Alexander Miller for providing useful references.

\bibliographystyle{plain}
\bibliography{arXiv-v3}
\end{document}